\documentclass[11pt,leqno]{article}
\usepackage[margin=1in]{geometry} 
\usepackage{amssymb,amsfonts,amsmath,bbm,mathrsfs,stmaryrd,mathtools}
\usepackage{xcolor}
\usepackage{url}

\usepackage{graphicx}

\usepackage{accents}

\usepackage{extarrows}

\usepackage[shortlabels]{enumitem}
\usepackage{tensor}

\usepackage{xr}
\usepackage[T1]{fontenc}
\usepackage[utf8]{inputenc}

\usepackage[colorlinks,
linkcolor=black!75!red,
citecolor=blue,
pdftitle={},
pdfproducer={pdfLaTeX},
pdfpagemode=None,
bookmarksopen=true,
bookmarksnumbered=true,
backref=page]{hyperref}

\usepackage{tikz}
\usetikzlibrary{arrows,calc,decorations.pathreplacing,decorations.markings,decorations.shapes,intersections,shapes.geometric,through,fit,shapes.symbols,positioning,decorations.pathmorphing}

\makeatletter
\newlength\zig@L
\newlength\zig@La
\newlength\zig@Lb

\newcommand{\xzigrightarrow}[2][]{%
  \mathrel{%
    \settowidth{\zig@La}{$\scriptstyle #2$}%
    \settowidth{\zig@Lb}{$\scriptstyle #1$}%
    \zig@L=\zig@La\relax
    \ifdim\zig@Lb>\zig@L \zig@L=\zig@Lb\fi
    \advance\zig@L by 2.2em\relax
    \tikz[baseline=-0.65ex]{%
      \draw[->,
            line cap=round,
            decorate,
            decoration={zigzag,segment length=4pt,amplitude=1.1pt}]%
        (0,0) -- (\zig@L,0)
        node[midway,above=2pt] {$\scriptstyle #2$}%
        \if\relax\detokenize{#1}\relax\else
          node[midway,below=2pt] {$\scriptstyle #1$}%
        \fi
      ;
    }%
  }%
}
\makeatother

\makeatletter
\newcommand{\squigjoin}{1mu} 

\def\sqleft@{\sim}                    
\def\sqmid@{\sim\mkern-\squigjoin}    

\def\rightsquigarrowfill@{%
  \arrowfill@{\sqleft@}{\sqmid@}{\mkern-4mu\succ}%
}

\newcommand{\xrightsquigarrow}[2][]{%
  \ext@arrow 0359\rightsquigarrowfill@{#1}{#2}%
}
\makeatother

\makeatletter
\newcommand*\circled[1]{\tikz[baseline=(char.base)]{
    \node[shape=circle, draw, inner sep=0pt, 
    minimum height={\f@size},] (char) {\vphantom{WAH1g}#1};}}
\makeatother

\makeatletter
\DeclareRobustCommand\widecheck[1]{{\mathpalette\@widecheck{#1}}}
\def\@widecheck#1#2{%
    \setbox\z@\hbox{\m@th$#1#2$}%
    \setbox\tw@\hbox{\m@th$#1%
       \widehat{%
          \vrule\@width\z@\@height\ht\z@
          \vrule\@height\z@\@width\wd\z@}$}%
    \dp\tw@-\ht\z@
    \@tempdima\ht\z@ \advance\@tempdima2\ht\tw@ \divide\@tempdima\thr@@
    \setbox\tw@\hbox{%
       \raise\@tempdima\hbox{\scalebox{1}[-1]{\lower\@tempdima\box
\tw@}}}%
    {\ooalign{\box\tw@ \cr \box\z@}}}
\makeatother

\usepackage{braket}

\usepackage[amsmath,thmmarks,hyperref]{ntheorem}
\usepackage{cleveref}

\newcommand\nthalias[1]{\AddToHook{env/#1/begin}{\crefalias{lemma}{#1}}}

\nthalias{definition}
\nthalias{example}
\nthalias{examples}
\nthalias{remark}
\nthalias{remarks}
\nthalias{convention}
\nthalias{notation}
\nthalias{construction}
\nthalias{sketch}
\nthalias{theoremN}
\nthalias{propositionN}
\nthalias{corollaryN}
\nthalias{lemma}
\nthalias{proposition}
\nthalias{corollary}
\nthalias{theorem}
\nthalias{conjecture}
\nthalias{question}
\nthalias{assumption}

\creflabelformat{enumi}{#2#1#3}

\crefname{section}{Section}{Sections}
\crefformat{section}{#2Section~#1#3} 
\Crefformat{section}{#2Section~#1#3} 

\crefname{subsection}{\S}{\S\S}
\AtBeginDocument{%
  \crefformat{subsection}{#2\S#1#3}%
  \Crefformat{subsection}{#2\S#1#3}%
}

\crefname{subsubsection}{\S}{\S\S}
\AtBeginDocument{%
  \crefformat{subsubsection}{#2\S#1#3}%
  \Crefformat{subsubsection}{#2\S#1#3}%
}

\theoremstyle{plain}

\newtheorem{lemma}{Lemma}[section]

\newtheorem{corollary}[lemma]{Corollary}
\newtheorem{theorem}[lemma]{Theorem}

\theoremstyle{plain}
\theoremnumbering{Alph}

\theoremstyle{plain}
\theorembodyfont{\upshape}
\theoremsymbol{\ensuremath{\blacklozenge}}

\newtheorem{example}[lemma]{Example}

\newtheorem{remark}[lemma]{Remark}

\crefname{definition}{definition}{definitions}
\crefformat{definition}{#2definition~#1#3} 
\Crefformat{definition}{#2Definition~#1#3} 

\crefname{ex}{example}{examples}
\crefformat{example}{#2example~#1#3} 
\Crefformat{example}{#2Example~#1#3} 

\crefname{exs}{example}{examples}
\crefformat{examples}{#2example~#1#3} 
\Crefformat{examples}{#2Example~#1#3} 

\crefname{remark}{remark}{remarks}
\crefformat{remark}{#2remark~#1#3} 
\Crefformat{remark}{#2Remark~#1#3} 

\crefname{remarks}{remark}{remarks}
\crefformat{remarks}{#2remark~#1#3} 
\Crefformat{remarks}{#2Remark~#1#3} 

\crefname{convention}{convention}{conventions}
\crefformat{convention}{#2convention~#1#3} 
\Crefformat{convention}{#2Convention~#1#3} 

\crefname{notation}{notation}{notations}
\crefformat{notation}{#2notation~#1#3} 
\Crefformat{notation}{#2Notation~#1#3} 

\crefname{table}{table}{tables}
\crefformat{table}{#2table~#1#3} 
\Crefformat{table}{#2Table~#1#3}

\crefname{lemma}{lemma}{lemmas}
\crefformat{lemma}{#2lemma~#1#3} 
\Crefformat{lemma}{#2Lemma~#1#3} 

\crefname{proposition}{proposition}{propositions}
\crefformat{proposition}{#2proposition~#1#3} 
\Crefformat{proposition}{#2Proposition~#1#3} 

\crefname{propositionN}{proposition}{propositions}
\crefformat{propositionN}{#2proposition~#1#3} 
\Crefformat{propositionN}{#2Proposition~#1#3} 

\crefname{corollary}{corollary}{corollaries}
\crefformat{corollary}{#2corollary~#1#3} 
\Crefformat{corollary}{#2Corollary~#1#3} 

\crefname{corollaryN}{corollary}{corollaries}
\crefformat{corollaryN}{#2corollary~#1#3} 
\Crefformat{corollaryN}{#2Corollary~#1#3} 

\crefname{theorem}{theorem}{theorems}
\crefformat{theorem}{#2theorem~#1#3} 
\Crefformat{theorem}{#2Theorem~#1#3} 

\crefname{theoremN}{theorem}{theorems}
\crefformat{theoremN}{#2theorem~#1#3} 
\Crefformat{theoremN}{#2Theorem~#1#3} 

\crefname{enumi}{}{}
\crefformat{enumi}{#2#1#3}
\Crefformat{enumi}{#2#1#3}

\crefname{assumption}{assumption}{Assumptions}
\crefformat{assumption}{#2assumption~#1#3} 
\Crefformat{assumption}{#2Assumption~#1#3} 

\crefname{construction}{construction}{Constructions}
\crefformat{construction}{#2construction~#1#3} 
\Crefformat{construction}{#2Construction~#1#3} 

\crefname{sketch}{sketch}{Sketches}
\crefformat{sketch}{#2sketch~#1#3} 
\Crefformat{sketch}{#2Sketch~#1#3} 

\crefname{question}{question}{Questions}
\crefformat{question}{#2question~#1#3} 
\Crefformat{question}{#2Question~#1#3} 

\crefname{equation}{}{}
\crefformat{equation}{(#2#1#3)} 
\Crefformat{equation}{(#2#1#3)}

\numberwithin{equation}{section}

\theoremstyle{nonumberplain}
\theoremsymbol{\ensuremath{\blacksquare}}

\newtheorem{proof}{Proof}
\newcommand\pf[1]{\newtheorem{#1}{Proof of \Cref{#1}}}

\newcommand\bC{{\mathbb C}}

\newcommand\bZ{{\mathbb Z}}

\newcommand\cL{{\mathcal L}}

\DeclareMathOperator{\Spec}{\mathrm{Spec}}

\newcommand{\comment}[1]{}

\title{Weak$^*$ polynomial density, diffuseness and proper infinitude}
\author{Alexandru Chirvasitu}

\begin{document}

\date{}

\newcommand{\Addresses}{{
  \bigskip
  \footnotesize

  \textsc{Department of Mathematics, University at Buffalo}
  \par\nopagebreak
  \textsc{Buffalo, NY 14260-2900, USA}  
  \par\nopagebreak
  \textit{E-mail address}: \texttt{achirvas@buffalo.edu}

}}

\maketitle

\begin{abstract}
  The unitary group of a von Neumann algebra on an $\aleph_0$-dimensional Hilbert space is polynomially weak$^*$-dense (equivalently, $G_{delta}$ dense or residual) in the normal unit ball if and only if the canonical finite atomic quotient of the von Neumann algebra is trivial. The analogous result holds for the usual weak$^*$ topology, in which case one can drop the normality constraint. In the latter form, this answers a question of T. Eisner and generalizes the counterpart for full bounded-operator algebras.
\end{abstract}


\vspace{.5cm}

\noindent{MSC 2020: 46L10; 46L51; 47C15; 28A25; 54E52; 32K12; 32H02; 46A32
  
}


\section*{Introduction}

The present note is motivated by a problem posed in passing in \cite{MR2769030}. Following confirmation \cite[Theorem 2.2]{MR2769030} that the unitary group $U(\ell^2)$ on an $\aleph_0$-dimensional Hilbert space is weakly \emph{residual} \cite[\S 8.A]{kech_descr} in the unit ball $\cL\left(\ell^2\right)_{\le 1}$ of the full algebra $\cL(\ell^2)$ of bounded operators, \cite[Remark 2.4]{MR2769030} suggests it may be of some interest to characterize those von Neumann algebras (realizable over separable Hilbert spaces) for which the unitary group is weakly$^*$ residual in the unit ball. One of the results below is that characterization:

\begin{theorem}\label{th:u.wast.dense.iff}
  The following conditions on a von Neumann algebra $M\le \cL(H)$ realizable on a separable Hilbert space are equivalent.
  \begin{enumerate}[(a),wide]
  \item\label{item:th:u.wast.dense.iff:g.delta} The unitary group $U(M)$ is a weak$^*$-dense $G_{\delta}$ subspace of the unit ball $M_{\le 1}$.    

  \item\label{item:th:u.wast.dense.iff:resid} $U(M)\subseteq M_{\le 1}$ is weak$^*$-residual. 

  \item\label{item:th:u.wast.dense.iff:dns} $U(M)\subseteq M_{\le 1}$ is weak$^*$-dense.

  \item\label{item:th:u.wast.dense.iff:sml.ball} Analogue of \Cref{item:th:u.wast.dense.iff:g.delta}, \Cref{item:th:u.wast.dense.iff:resid} or \Cref{item:th:u.wast.dense.iff:dns}, with the normal subset $M_{\le 1, n}$ of the unit ball in place of $M_{\le 1}$.

  \item\label{item:th:u.wast.dense.iff:can.quot} The canonical finite atomic quotient of $M$ vanishes. 
  \end{enumerate}
\end{theorem}

Recall also \cite[Definition 2.1]{MR3101250}'s \emph{weak polynomial topology} (not even linear, let alone locally convex: \Cref{ex:wp.not.add}): that of weak convergence for arbitrary powers:
\begin{equation*}
  x_{\lambda}
  \xrightarrow[\quad\lambda\quad]{\quad\text{polynomial weak}\quad}
  x
  \quad
  \iff
  \quad
  \forall\left(n\in \bZ_{\ge 0}\right)
  \left(
    x^n_{\lambda}
    \xrightarrow[\quad\lambda\quad]{\quad\text{weak}\quad}
    x^n
  \right).
\end{equation*}
The topology dominates the weak, and on bounded subsets is intermediate between it and the strong. \Cref{th:u.wast.dense.iff} has an analogue (albeit a partial one) in that context (where we naturally substitute ``weak$^*$'' for ``weak'', though it will make no difference on bounded sets once von Neumann algebras are realized spatially on Hilbert spaces).

\begin{theorem}\label{th:u.wast.dense.iff.pw}
  The following conditions on a von Neumann algebra $M\le \cL(H)$ realizable on a separable Hilbert space are equivalent.
  \begin{enumerate}[(a),wide]
  \item\label{item:th:u.wast.dense.iff.pw:g.delta} The unitary group $U(M)$ is a polynomially weak$^*$-dense $G_{\delta}$ subspace of the normal unit ball $M_{\le 1,n}$.    

  \item\label{item:th:u.wast.dense.iff.pw:resid} $U(M)\subseteq M_{\le 1,n}$ is polynomially weak$^*$-residual. 

  \item\label{item:th:u.wast.dense.iff.pw:dns} $U(M)\subseteq M_{\le 1,n}$ is polynomially weak$^*$-dense.

  \item\label{item:th:u.wast.dense.iff.pw:can.quot} The canonical finite atomic quotient of $M$ vanishes. 
  \end{enumerate}
\end{theorem}

\subsection*{Acknowledgments}

I am grateful for helpful remarks from T. Eisner on an early version of the draft. 


\section{Large $G_{\delta}$ unitary groups in von Neumann algebras}\label{se:lg.unit}

In the context of \Cref{th:u.wast.dense.iff}, \emph{density} is easily characterized; in stating \Cref{le:dns.iff.no.fin.at} we remind the reader some of the standard background on canonical $W^*$-algebra decompositions.
\begin{itemize}[wide]
\item There is a unique central splitting \cite[\S I.6.7, Corollary 1]{dixw} $M\cong M_{<\infty}\times M_{\text{pr}\infty}$ with the left- (right-)hand factor \emph{finite} (respectively \emph{properly infinite}). 

\item There is also a unique central splitting \cite[\S 10.21]{strat} $M\cong M_{a}\times M_{\not a}$ into an \emph{atomic} and a \emph{non-atomic} (or \emph{diffuse} \cite[\S 29.2]{strat}) component: the former means that every non-zero projection dominates a non-zero minimal projection, while the latter is the strong negation to the effect that there are no minimal non-zero projections. 
\end{itemize}
Being central, these can be combined into the obvious finer (four-factor) decomposition. 

\begin{lemma}\label{le:dns.iff.no.fin.at}
  The unitary group of a $W^*$-algebra is weak$^*$-dense in its unit ball if and only if its canonical finite atomic quotient vanishes. 
\end{lemma}
\begin{proof}
  The harder direction (density assuming vanishing) is a conjunction of \cite[Theorem 1]{MR52695} (density given non-atomicity) and \cite[Theorem 5]{MR288588} (density given proper infinitude), and the converse is immediate: a non-vanishing quotient will have a finite-matrix-algebra factor, and the density property (which certainly does not hold for $M_n$) plainly survives the passage to central quotients. 
\end{proof}

This suffices to settle the question advertised above.

\pf{th:u.wast.dense.iff}
\begin{th:u.wast.dense.iff}
  Plainly,
  \begin{equation*}
    \text{\Cref{item:th:u.wast.dense.iff:g.delta}}
    \xRightarrow{\quad\text{formal}\quad}
    \text{\Cref{item:th:u.wast.dense.iff:resid}}
    \xRightarrow{\quad\text{formal}\quad}
    \text{\Cref{item:th:u.wast.dense.iff:dns}}
    \xLeftrightarrow{\quad\text{\Cref{le:dns.iff.no.fin.at}}\quad}
    \text{\Cref{item:th:u.wast.dense.iff:can.quot}}.
  \end{equation*}
  To verify \Cref{item:th:u.wast.dense.iff:dns} $\Rightarrow$ \Cref{item:th:u.wast.dense.iff:g.delta}:
  \begin{itemize}[wide]
  \item $U(M)$ is \emph{Polish} \cite[\S IV.8]{tak1} (for it is so in the \emph{strong$^*$ topology} as noted in \cite[proof of Theorem IV.8.28]{tak1}, and said topology is also weak/weak$^*$ \cite[Remark II.4.10]{tak1} on $U(H)$): loc. cit. refers to the plain weak topology, coinciding with the weak$^*$ on bounded sets \cite[\S I.3.2]{dixw};
    
  \item so also $G_{\delta}$ in the again Polish unit ball $M_{\le 1}\subseteq \cL(H)_{\le 1}$ by \cite[Theorem A.1]{tak1}.
  \end{itemize}
  The displayed conditions thus being equivalent, \Cref{item:th:u.wast.dense.iff:sml.ball} joins the cluster.
\end{th:u.wast.dense.iff}

\begin{example}\label{ex:wp.not.add}
  Consider the operators $A_n,B_n\in \cL\left(\ell^2(\bZ_{\ge 0})\right)$, expressed as matrices in the standard basis $(e_n)_{n\ge 0}$, containing a single non-zero entry in position $(0,n)$ (respectively $(n,0)$). We have
  \begin{equation*}
    A_n,B_n
    \xrightarrow[\quad n\quad]{\quad\text{weak$^*$ polynomially}\quad}
    0,
  \end{equation*}
  while $(A_n+B_n)^2$ converges weakly to the projection with a single unit entry in position $(0,0)$. 
\end{example}

We tackle the issue of polynomial weak$^*$ ($G_{\delta}$) density separately for several classes of von Neumann algebras.

\begin{theorem}\label{th:wp.when.untr.dns.prop.inf}
  For a properly infinite $W^*$-algebra $M\le \cL(H)$, $\dim H=\aleph_0$ the unitary group $U(M)\subseteq M_{\le 1}$ polynomially weak$^*$ $G_{\delta}$ dense.
\end{theorem}
\begin{proof}
  Being $G_{\delta}$ transports to the stronger topology, so the one point left to settle is density assuming proper infinitude. The matricial techniques applicable \cite[Theorem 1]{zbMATH03749579} to $\cL(H)$ transport over well to the properly infinite context. Identify $M\cong M\otimes \cL(\ell^2)$ (cf. the proof of \cite[Theorem 5]{MR288588}), by selecting a sequence of projections and partial isometries
  \begin{equation*}
    u_{nn}=p_n\sim 1
    ,\quad 
    u_{mn}
    ,\quad
    m,n\in \bZ_{\ge 0}
    ,\quad
    u^*_{nm}=u_{mn}
    ,\quad
    u_{mn}u_{pq}=\delta_{np}\delta_{mq} p_m
  \end{equation*}
  with ``$\sim$'' denoting the usual notion \cite[Definition V.1.2]{tak1} of projection \emph{equivalence} in a $W^*$-algebra. 

  As any contraction $x\in p_0 M p_0\cong M$ admits isometric \emph{dilations}
  \begin{equation*}
    x_n:=x+u_{n0}(1-x^*x)^{1/2} + \sum_{m\ge 1}u_{n+m,m}      
  \end{equation*}
  with
  \begin{equation*}
    p_0 x_np_0
    \xrightarrow[\quad n\quad]{\quad\text{weak$^*$ polynomially}\quad}
    x,
  \end{equation*}
  the dilation/weak$^*$-density connection noted in \cite[Problem 224]{hal_hspb_2e_1982} proves isometries weak$^*$ polynomially dense in $M_{\le 1}$; this reduces the problem to proving unitaries dense in the set of isometries instead. This, however, holds in any von Neumann algebra whatsoever, with respect to the strong$^*$ topology: the usual $\cL(H)$ proof (e.g. \cite[Solution 225]{hal_hspb_2e_1982} or \cite[Theorem 1, step (III)]{zbMATH03749579}) uses the \emph{Wold decomposition} \cite[\S 4.3, Exercise (6)]{arv_spec} $V=V_u\oplus V_{p}$ of an arbitrary isometry as a direct sum of a unitary $V_u$ and a sum $V_p$ (a \emph{pure} isometry) of unilateral shifts on $\ell^2$, and that decomposition is internal to any von Neumann algebra $M\ni V$. Indeed,
  \begin{equation*}
    \left(\text{projection on }\mathrm{Im}~V_u\right)
    =
    \lim_n^{\text{weak$^*$}}V^n V^{*n}.
  \end{equation*}
  This concludes the proof. 
\end{proof}

\begin{remark}\label{re:unitary.power-dil}
  \Cref{th:wp.when.untr.dns.prop.inf} can ultimately be traced back to \cite[Th\'eor\`eme I]{zbMATH03132193}, ensuring the existence of \emph{unitary power dilations} \cite[Problem 227]{hal_hspb_2e_1982} for contractive operators on infinite-dimensional Hilbert spaces:
  \begin{equation*}
    \forall\left(A\in \cL(H)\right)
    \left(
      \|A\|\le 1
      \xRightarrow{\quad}
      \exists\left(\text{unitary }U,\ \text{projection }P\in \cL(K)\right)
      \left(
        \forall n
        \left(
          A^n=PU^nP
        \right)
      \right)
    \right)
  \end{equation*}
  for some Hilbert space $K\ge H$. Weak polynomial approximability by unitaries follows from the usual dilation/weak-closure link (\cite[Problem 224]{hal_hspb_2e_1982}, technique implicit in \cite[proof of Theorem 5]{MR288588}).
\end{remark}

There is also a type-I version of \Cref{th:wp.when.untr.dns.prop.inf}, albeit requiring different methods. 

\begin{theorem}\label{th:wp.when.untr.dns.t1}
  The unitary group $U(M)$ of a type-I von Neumann algebra $M\le \cL(H)$, $\dim H=\aleph_0$ is polynomially weak$^*$-dense (equivalently, dense $G_{\delta}$) in $M_{\le 1}$ if and only if the finite canonical summand of $M$ is non-atomic. 
\end{theorem}
\begin{proof}
  The equivalence between density and $G_{\delta}$ density has already been noted, and the forward implication $(\Rightarrow)$ follows from \Cref{th:u.wast.dense.iff}. It thus remains to prove density assuming $M$ type-I non-atomic, as well as finite (having disposed of the properly infinite canonical summand in \Cref{th:wp.when.untr.dns.prop.inf}).  

  The finiteness assumption gives a decomposition \cite[Theorem V.1.27]{tak1}
  \begin{equation*}
    M\cong \prod_{n\in \bZ_{>0}} A_n\otimes M_n(\bC)
    ,\quad
    A_n\text{ abelian}
    \quad
    \left(\text{product in the category of $W^*$-algebras}\right),
  \end{equation*}
  with $A_n\cong C(X_n)$ (continuous functions on the respective spectra $X_n:=\Spec(A_n)$) for \emph{hyperstonean} \cite[Theorem III.1.18]{tak1} spaces $X_n$: compact Hausdorff, with enough \emph{normal} measures (``normal'' meaning assigning nowhere dense sets mass 0 and ``enough'' as in such measures having trivial joint kernel in $C(X)$). The individual factors can be treated separately, so we may as well have
  \begin{equation*}
    M
    \cong
    C(X)\otimes M_n
    \cong
    C(X,M_n)
    :=
    \left\{
      \text{hyperstonean }X\xrightarrow[\quad\text{continuous}\quad]{\quad}M_n:=M_n(\bC)
    \right\}
  \end{equation*}
  throughout. The diffuseness assumption means precisely that $X$ is perfect (has no isolated points).

  An arbitrary element $a\in C(X,M_n)$, to be approximated polynomially weak$^*$ by unitaries, may and will be assumed constant with value again denoted by $a\in M_{m,\le 1}$ (by a slight notational abuse): it can always be uniformly approximated by locally constant elements by partitioning
  \begin{equation*}
    X=\bigsqcup_{i=1}^N X_i
    ,\quad
    \forall \left(1\le i\le N\right)
    \left(\mathrm{diam}~a(X_i)<\varepsilon\right)
    ,\quad
    X_i=\overline{X}_i=\accentset{\large\circ}{X}_i\subseteq X
  \end{equation*}
  for arbitrarily small $\varepsilon$, and working individually on each $X_i$. We will also fix normal states $\psi_j$, $1\le j\le J$ specifying the polynomial weak$^*$ neighborhood
  \begin{equation*}
    V_{\left(\psi_j\right)_j, m, \varepsilon}(a)
    :=
    \left\{
      g\in C(X,M_n)
      \ :\
      \forall\left(1\le j\le J\right)
      \forall\left(0\le k\le m\right)
      \left(\left|\psi_j(g^k)-\psi_j(a^k)\right|<\varepsilon\right)
    \right\}
  \end{equation*}
  of interest, necessarily of the form
  \begin{equation*}
    C(X,M_n)
    \ni
    f
    \xmapsto{\quad \psi_j\quad}
    \int_{X}\mathrm{tr}~f(x)\ \mathrm{d}\nu_j(x)
    \in
    \bC
  \end{equation*}
  for normal measures $\nu_j\in M(X,M_n)$ valued in the space of \emph{density $n\times n$ matrices} (trace-1 positive). We further simplify the setting in a number of ways.

  \begin{enumerate}[(a),wide]
  \item The density-valued measures $\nu_i$ can be norm-approximated arbitrarily well by simple counterparts
    \begin{equation*}
      \nu'_{i}=\sum_{j=1}^{k_i}\mu_{ij}q_{ij}
      ,\quad
      \begin{gathered}
        \mu_{ij}\text{ concentrated on }
        X_{ij}=\overline{X}_{ij}=\accentset{\large\circ}{X}_{ij}\subseteq X\\
        \text{density matrix }q_{ij}\in M_n;
      \end{gathered}
    \end{equation*}
    consequently, restricting attention to individual clopens, we can assume $\nu_j=\mu_j\otimes q_j$ simple. 
    
  \item Next, the $\mu_j$ can all be assumed $\mu$-continuous for a probability measure $\mu:=\mu_1$ (for $\left(\sum_j \mu_j\right)\otimes \left(\sum_j q_j\right)$ can always be appended to the family $\{\nu_j\}_j$) so that $\mathrm{d}\mu_j = f_j\mathrm{d}\mu$ for $f_j\in L^1(\mu)$ \cite[Theorem III.10.2]{ds_linop-1_1958}.

  \item Finally, $f_j$ may as well be step functions, for they are uniformly approximable by $C$-bounded step functions away from a set $Y\subseteq X$ with $\mu(Y)<\varepsilon$ for convenient cutoffs $\varepsilon,C>0$. For our purposes, restricting once again to convenient clopen subsets, the $f_j$ may all as well be constant.
  \end{enumerate}
  We seek unitary approximants among step functions again:
  \begin{equation*}
    u=
    \sum_{i=1}^N \chi_{X_i} u_i
    ,\quad
    \begin{gathered}
      \chi_{\bullet}:=\text{characteristic function of $\bullet$}\\
      X_{i}=\overline{X}_{i}=\accentset{\large\circ}{X}_{i}\subseteq X\\
      \forall\left(1\le i\le N\right)\left(u_i\in U(n)\right),\quad
    \end{gathered}
  \end{equation*}
  so that the desired approximation amounts to
  \begin{equation*}
    \forall\left(0\le k\le m\right)
    \forall\left(0\le j\le J\right)
    \left(\left|\sum_i \mu\left(X_i\right)\mathrm{tr}~u^k_i q_j-\mathrm{tr}~a^k q_j\right|<\varepsilon\right).
  \end{equation*}
  Given diffuseness again, the family of coefficients $\mu(X_i)$ can be arbitrary subject to the constraint that it constitute a convex combination. The conclusion, then, would follow assuming that
  \begin{equation*}
    \left(1,a,\cdots,a^m\right)
    \in
    \text{convex hull }
    \mathrm{cvx}~\left\{\left(1,u,\cdots,u^m\right)\ :\ u\in U(n)\right\}.
  \end{equation*}
  To conclude, recall that for \emph{any} holomorphic map $M_{n,\le 1}\xrightarrow{f}\bC^d$ defined on (a neighborhood of) the $n\times n$-matrix unit ball we have
  \begin{equation*}
    f\left(M_{n,\le 1}\right)
    \subseteq
    \text{convex hull }
    \mathrm{cvx}~f\left(\text{unitary group }U(n)\right)
  \end{equation*}
  by, say, \cite[Theorem 9 and Proposition 2(a)]{MR407330}. 
\end{proof}

\Cref{th:wp.when.untr.dns.t1} ensures partial polynomial weak$^*$ density in the unit ball under the sensible hypotheses. 

\begin{corollary}\label{cor:norm.pw.approx}
  For a von Neumann algebra $M\le \cL(H)$, $\dim H=\aleph_0$ with vanishing finite atomic factor the normal operators in $M_{\le 1}$ are contained in the polynomial weak$^*$ closure of $U(M)$. 
\end{corollary}
\begin{proof}
  These generate abelian von Neumann subalgebras of $M$. Observe that every \emph{maximal} such $A\le M$ (with respect to inclusion) will be non-atomic again: were $0\ne p\in A$ minimal therein, $A=Ap\oplus A(1-p)$ and every $0\ne q< p$ strictly dominated by $p$ commutes with $A$. This reduces the problem to abelian $M$, case covered by \Cref{th:wp.when.untr.dns.t1}. 
\end{proof}

\begin{remark}\label{re:rec.dye}
  \Cref{th:wp.when.untr.dns.t1} being in place, one can recover Dye's result \cite[Theorem 1]{MR52695} on the weak$^*$ density of the unitary group in the unit ball of a diffuse von Neumann algebra by essentially the same device \cite[Lemma 2.4]{MR52695} employs (not operative in the polynomial weak$^*$ setting): \emph{polar decomposition} $a=u|a|$ \cite[Appendix III, \S 3]{dixw} for an arbitrary $a\in M$ reduces the problem to density in the normal unit ball, hence the conclusion by restricting attention to a(n automatically diffuse) maximal abelian self-adjoint algebra containing $|a|$.
\end{remark}

\pf{th:u.wast.dense.iff.pw}
\begin{th:u.wast.dense.iff.pw}
  To address the one non-trivial implication \Cref{item:th:u.wast.dense.iff.pw:can.quot} $\Rightarrow$ \Cref{item:th:u.wast.dense.iff.pw:g.delta}: \Cref{th:wp.when.untr.dns.prop.inf} reduces the problem to the finite non-atomic case, whence the conclusion by \Cref{th:wp.when.untr.dns.t1} applied to an (again diffuse) maximal abelian self-adjoint subalgebra containing a given normal element. 
\end{th:u.wast.dense.iff.pw}


\addcontentsline{toc}{section}{References}

\def\polhk#1{\setbox0=\hbox{#1}{\ooalign{\hidewidth
  \lower1.5ex\hbox{`}\hidewidth\crcr\unhbox0}}}


\Addresses

\end{document}